\documentclass[11pt]{amsart}
\usepackage{a4wide}

\newcommand{\e}{\varepsilon}
\newcommand{\IN}{\mathbb N}
\newcommand{\IR}{\mathbb R}

\newcommand{\Ra}{\Rightarrow}

\newtheorem{theorem}{Theorem}
\newtheorem{corollary}{Corollary}

\newtheorem{question}{Question}

\title[Approximating points of a Banach space by points of an operator image]{Approximating points of a Banach space\\ by points of an operator image}
\author{Taras Banakh and Yuriy Golovaty}
\address{Ivan Franko National University of Lviv, Ukraine}
\email{t.o.banakh@gmail.com, yuriy.golovaty@lnu.edu.ua}
\dedicatory{Dedicated to the 60th birthday of M.M. Zarichnyi and 70th birthday of A.M.Plichko}

\begin{document}

\begin{abstract}Answering one problem that has its origins in quantum mechanics, we prove that for any sequence $(A_n)_{n\in\IN}$ of convex nowhere dense sets in a Banach space $X$ and any sequence $(\e_n)_{n=1}^\infty$ of positive real numbers with $\lim_{n\to\infty}\e_n=0$, the set $A=\{x\in X:\forall n\in\IN\;\exists a\in A_n\;\;\|x-a\|< \e_n\}$ is nowhere dense in $X$.
\end{abstract}
\maketitle

The question that is considered in the article arose in a problem of quantum mechanics. 
In the last two decades the Hamiltonians with singular potentials supported on submanifolds of  the configuration space $\mathbb{R}^d$ of a lower dimension, also known as pseudo-Hamiltonians, have attracted considerable attention both in the physical and mathematical literature. The potentials that are distributions with supports on curves, surfaces, and more complicated sets composed of them, often used in simulation of quantum systems, because the corresponding Schr\"{o}dinger equations are generally easier to solve.  These so-called exactly solvable models allow us to calculate explicitly numerical characteristics of systems such as eigenvalues, eigenfunctions or scattering data, the original differential equation being reduced to the analysis of an algebraic or functional problem. Very often the pseudo-Hamiltonians reveal an unquestioned effectiveness whenever the exact solvability together with non trivial qualitative description of  the actual quantum dynamics is required.
In spite of all advantages of the exactly solvable models they  give rise to many mathematical difficulties. One of the main difficulties deals with the multiplication of distributions.  To get round  the problem  of multiplication of distributions, we can regularize pseudo-potential $V\in \mathcal{D}'(\mathbb{R}^d)$ by a sequence of smooth enough potentials $V^\e$ such that $V^\e\to V$ as $\e\to 0$ in the sense of distributions,  and then investigate the convergence of Hamiltonians $H_\e=-\Delta+V^\e$ in a suitable operator topology \cite{GolovatyHrynivProcEdinburgh2013}--\cite{GolovatyJPA2018}. The main goal is to find the limit self-adjoint operator  $H_0$  and thereby to assign for the quantum system a mathematically correct  solvable model that describes the real quantum evolution with adequate accuracy.

Let $M$ be a smooth compact  manifold embedded in $\mathbb{R}^d$. Assume that $V_M\in \mathcal{D}'(\mathbb{R}^d)$ and  ${\rm supp}\, V_M\subset M$. We choose a sequence $\{V^\e\}_{\e>0}$ of smooth functions with compact supports shrinking to the manifold $M$ as $\e\to 0$. Also, this sequence converges to the distribution $V_M$ in $\mathcal{D}'(\mathbb{R}^d)$. Let us introduce the  sesquilinear form
\begin{equation*}
  a_\e(u,v)=\int_{\mathbb{R}^d}\big( \nabla u \nabla \bar{v}+V^\e(x)u\bar{v}\big)\,dx
\end{equation*}
in the Sobolev space $W_2^1(\mathbb{R}^d)$. We can realize the Hamiltonian as the operator $A_\e$ associated with form $a_\e$, i.e., $a_\e(u,v)=(A_\e u,v)_{L_2(\mathbb{R}^d)}$. Two cases arise depending on the order of the distribution $V_M$. For example, if $V_M$ is a $\delta_M$-measure with density $\mu$, that is to say
\begin{equation*}
  V_M(\phi)=\int_{M}\mu \phi\,dM, \qquad \phi\in C_0^\infty(\mathbb{R}^d),
\end{equation*}
then the forms $a_\e$ are bounded from below uniformly with respect to $\e$ and there exists the limit form
\begin{equation*}
  a_0(u,v)=\int_{\mathbb{R}^d} \nabla u \nabla \bar{v}\,dx+\int_{M}\mu u\bar{v}\,dM.
\end{equation*}
with the same domain $W_2^1(\mathbb{R}^d)$. 
From the convergence of the forms  we readily deduce the convergence  $A_\e\to A_0$  in the strong resolvent topology, where $A_0$ is an operator associated with $a_0$. 
In the case when the distribution $V_M$ is more singular,
the forms $a_\e$ are not uniformly bounded from below and the presupposed ``limit form'' $a_0$ has generally  the domain which does not coincide with the domain of $a_\e$. For instance, when trying to prove the operator convergence in the problem with
$V_M=\partial_\nu \delta_M$, where $\partial_\nu$ is a normal derivative on $M$,  we were confronted with

\begin{question}\label{quest} Is it true that for any positive real number $s$ there exist positive real numbers  $C,\alpha,\beta$ such that for any function $v\in W_2^{-s}(M)$ there exists a sequence  $\{v_n\}_{n=1}^\infty\subset W_2^{s}(M)$ such that $$\|v-v_n\|_{W_2^{-s}(M)}\leq C \cdot n^{-\alpha}\mbox{ \ and \ }\|v_n\|_{W_2^{s}(M)}\leq C\cdot n^{\beta}$$
for all $n\in\IN$?
\end{question}

It turns out that the answer to this question is negative. This negative answer will be derived (in Corollary~\ref{cor2})  from the following  theorems.

\begin{theorem}\label{t:main}  For any sequence $(A_n)_{n\in\IN}$ of convex nowhere dense sets in a normed space $X$ and any sequence $(\e_n)_{n\in\IN}$ of positive real numbers with $\lim_{n\to\infty}\e_n=0$ the set $$A=\{x\in X:\forall n\in\IN\;\;\exists a\in A_n\;\;\;\|x-a\|<\e_n\}$$ is convex and nowhere dense in $X$.
\end{theorem}

\begin{proof} Let $B=\{x\in X:\|x\|<1\}$ be the open unit ball in the normed space $X$ and observe that
$$A=\bigcap_{n\in\IN}(A_n+\e_n B),$$
which implies that the set $A$ is convex (as the intersection of convex sets $A_n+\e_n B$).

It remains to prove that the set $A$ is nowhere dense.
In the opposite case its closure $\bar A$ contains an $\e$-ball $c+\e B$ for some small $\e>0$. Since $\lim_{n\to\infty}\e_n=0$, there exists $n\in\IN$ such that $\e_n<\frac18\e$. It follows that $$c+\e B\subset\bar A\subset\overline{A_n+\e_n B}\subset A_n+2\e_n B.$$ Then $\e B\subset (A_n-c)+2\e_n B$.
Since the convex set $A_n-c$ is nowhere dense in $X$, there exists a point $b\in \frac14\e B\setminus \overline{A_n-c}$. By the Hahn-Banach Separation Theorem, there exists an $\mathbb R$-linear functional $x^*:X\to\mathbb R$ of norm $\|x^*\|=1$ such that $$\sup x^*(A_n-c)<x^*(b)\le\|x^*\|\cdot\|b\|<\tfrac14\e.$$ By the definition of the norm $\|x^*\|$ of the functional $x^*$, there exists a point $x\in B$ such that $x^*(x)>\frac12$. Since $\e x\in \e B\subset (A_n-c)+2\e_nB$, there exist points $a\in A_n$ and $z\in B$ such that $\e x=a-c+2\e_n z$. Then
\begin{multline*}
\tfrac12\e<\e\cdot  x^*(x)=x^*(\e x)=x^*(a-c+2\e_nx)=x^*(a-c)+2\e_n\cdot x^*(z)\le\\
\le \sup x^*(A_n-c)+2\e_n\cdot\|x^*\|\cdot\|z\|
\le x^*(b)+2\e_n<\tfrac14\e+\tfrac14\e=\tfrac12\e,
\end{multline*}
which is a contradiction that completes the proof of the theorem.
\end{proof}

It is interesting that Theorem~\ref{t:main} does not generalize to locally convex linear metric spaces. Moreover, the property described in Theorem~\ref{t:main} can be used to characterize normable spaces among metrizable locally convex spaces.

By a {\em locally convex space} we understand a locally convex linear topological space over the field of real numbers. A locally convex space is {\em normable} if its topology is generated by a  norm. By Proposition 4.12 in \cite{FA}, a locally convex space is normable if and only if it contains a bounded neighborhood of zero.

A subset $B$ of a linear topological space $X$ is {\em bounded} if for any neighborhood $U$ of zero in $X$ there exists a positive real number $r$ such that $B\subset r{\cdot}U$.

\begin{theorem}\label{t2} Let $X$ be a locally convex space and $(U_n)_{n\in\IN}$ be a base of neighborhoods of zero in $X$. Then the following conditions are equivalent:
\begin{enumerate}
\item $X$ is normable;
\item for any sequence $(A_n)_{n\in\IN}$ of nowhere dense convex sets in $X$, the intersection\newline $\bigcap_{n\in\IN}(A_n+U_n)$ is nowhere dense in $X$;
\item for any sequence $(L_n)_{n\in\IN}$ of nowhere dense linear subspaces in $X$ the intersection $\bigcap_{n\in\IN}(L_n+U_n)$ is not equal to $X$.
\end{enumerate}
\end{theorem}

\begin{proof} $(1)\Ra(2)$ Assume that the locally convex space $X$ is normable and let $\|\cdot\|$ be a norm generating the topology of $X$. Since $(U_n)_{n\in\IN}$ is a base of neighborhoods of zero, for every $k\in\IN$ there exists $n_k\in\IN$ such that $U_{n_k}\subset \{x\in X:\|x\|<\tfrac1k\}$.

Let $(A_n)_{n\in\IN}$ be a sequence of nowhere dense convex sets in $X$. Applying Theorem~\ref{t:main} to the normed space $(X,\|\cdot\|)$, we conclude that the set $$A=\{x\in X:\forall k\in\IN\;\exists y\in A_{n_k}\;\;\|x-y\|<\tfrac1k\}$$is nowhere dense in $X$. Observing that
$$\bigcap_{n\in\IN}(A_n+U_n)\subset \bigcap_{k\in\IN}(A_{n_k}+U_{n_k})\subset A$$we conclude that the set $\bigcap_{n\in\IN}(A_n+U_n)$ is nowhere dense, too.
\smallskip

The implication $(2)\Ra(3)$ is trivial.
\smallskip

$(3)\Ra(1)$ Assume that the space $X$ is not normable. By Proposition 4.12 \cite{FA}, the space contains no bounded neighborhoods of zero. Then for every $n\in\IN$ the neighborhood $V_n=U_n\cap (-U_n)$ of zero is unbounded. By Theorem 3.18 in \cite{Rudin}, the set $V_n$ is not weakly bounded, which allows us to find a linear continuous functional $f_n:X\to\IR$ such that the image $f_n(V_n)$ is unbounded in the real line. Taking into account that $V_n$ is convex and $V_n=-V_n$, we conclude that $f_n(V)=\IR$. Then for the nowhere dense  linear subspace $L_n=f^{-1}_n(0)$ of $X$ we get $X=L_n+V_n\subset L_n+U_n$, which implies $\bigcap_{n\in\IN}(L_n+U_n)=X$.
\end{proof}

In spite of Theorem~\ref{t2}, Theorem~\ref{t:main} does admit a partial generalization to locally convex linear metric spaces.

\begin{theorem}\label{t3} Let $X$ be a locally convex space and $d$ be an invariant metric generating the topology of $X$. For any sequence $(B_n)_{n\in\IN}$ of nowhere dense bounded convex sets in $X$ and any sequence $(\e_n)_{n\in\IN}$ of positive real numbers with $\lim_{n\to\infty}\e_n=0$ the set $$B=\{x\in X:\forall n\in\IN\;\;\exists y\in B_n\;\;\;d(x,y)<\e_n\}$$ is nowhere dense in $X$.
\end{theorem}

\begin{proof} The space $X$ being locally convex and metrizable, has a neighborhood base $\{U_k\}_{k\in\IN}$ at zero consisting of open convex neighborhoods of zero such that $U_1=X$ and $U_{k+1}\subset U_k=-U_k$ for all $k\in\IN$.
For every $n\in\IN$ let $k_n\in\IN$ be the largest number such that $\{x\in X:d(x,0)<\e_n\}\subset U_{k_n}$. It follows from $\lim_{n\to\infty}\e_n=0$ that $\lim_{n\to\infty}k_n=\infty$.

Observe that $B\subset \bigcap_{n\in\IN}(B_n+U_{k_n})$. So, it suffices to prove that the set $C=\bigcap_{n\in\IN}(B_n+U_{k_n})$ is nowhere dense. It is clear that the set $C$ is convex (being the intersection of the convex sets $B_n+U_{k_n}$). Next, we show that the set  $C$ is bounded in $X$. Given any neighborhood $U\subset X$ of zero, find $n\in\IN$ such that $U_{k_n}\subset U$. Such number $k_n$ exists as $\lim_{i\to\infty}k_i=\infty$ and $\{U_k\}_{k\in\IN}$ is a decreasing neighborhood base at zero. Since the set $B_n$ is bounded, there exists a real number $r$ such that $B_n\subset r{\cdot}U_{k_n}$. The convexity of $U_{k_n}$ ensures that for any $x,y\in U_{k_n}$ we have $$rx+y=(r+1)(\tfrac{r}{r+1}x+\tfrac1{r+1}y)\in (r+1){\cdot}U_{k_n}$$and hence
$C\subset B_n+U_{k_n}\subset r{\cdot}U_{k_n}+U_{k_n}=(r+1){\cdot}U_{k_n}\subset(r+1){\cdot}U$, which means that the set $C$ is bounded.

Assuming that $C$ is not nowhere dense, we conclude that its closure $\bar C$ has non-empty interior and then  $\bar C-\bar C:=\{x-y:x,y\in\bar C\}$ is a bounded convex symmetric neighborhood of zero in $X$. By Proposition 4.12 in \cite{FA}, the locally convex space $X$ is normable. By the implication $(1)\Ra(2)$ in Theorem~\ref{t2}, the intersection $\bigcap_{n\in\IN}(B_n+U_{n_k})\supset B$ is nowhere dense in $X$.
\end{proof}

Now we shall use Theorem~\ref{t3} to give a negative answer to Question~\ref{quest}.

\begin{corollary}\label{cor1} Let $T:X\to Y$ be a non-open bounded operator from a Banach space $(X,\|\cdot\|_X)$ to a locally convex linear metric space $(Y,\|\cdot-\cdot\|_Y)$. For any sequences  $(r_n)_{n=1}^\infty$ and $(\e_n)_{n=1}^\infty$ of positive real numbers with $\lim_{n\to\infty}\e_n=0$, the set $$A=\big\{y\in Y:\forall n\in\IN\;\;\exists x\in X\;\;\;\big(\,\|x\|_X<r_n\mbox{ and }\|y-Tx\|_Y<\e_n\big)\big\}$$ is convex and nowhere dense in $Y$.
\end{corollary}

\begin{proof} For every $n\in\IN$ consider the $\e_n$-neighborhood $U_n=\{y\in Y:\|y-0\|_Y<\e_n\}$ of zero in $Y$. Since the operator $T$ is not open, the image $T(B_X)$ of the unit ball $B_X=\{x\in X:\|x\|_X<1\}$ has empty interior in  $Y$. We claim that $T(B_X)$ is nowhere dense in $Y$. If the locally convex space $Y$ is not normable, then $Y$ contains no bounded neighborhoods of zero, which implies that the bounded set $T(B_X)$ is nowhere dense in $Y$. If $Y$ is normable, then the set $T(B_X)$ is nowhere dense in $Y$ by Banach's Lemma 2.23 in \cite{FA}.

 Then for every $n\in\IN$ the bounded convex set $A_n:=T(r_n B_X)$ is nowhere dense in $Y$. Applying Theorem~\ref{t3}, we conclude that the set $A=\bigcap_{n\in\IN}(A_n+U_n)$ is convex and nowhere dense in $Y$.
\end{proof}

\begin{corollary}\label{cor2} Let $T:X\to Y$ be a non-open bounded operator from a Banach space $(X,\|\cdot\|_X)$ to a locally convex linear metric space $(Y,\|\cdot-\cdot\|_Y)$. For any positive real constants $C,\alpha,\beta$ the set $$A_{C,\alpha,\beta}=\big\{y\in Y:\forall n\in\IN\;\;\exists x\in X\;\;\big(\,\|y-Tx\|_Y<Cn^{-\alpha}\mbox{ \ and \ }\|x\|_X< Cn^\beta\,\big)\big\}$$ is convex and nowhere dense in $Y$. Consequently, the set $$A=\bigcup_{C,\alpha,\beta>0}A_{C,\alpha,\beta}=\bigcup_{k\in\IN}A_{k,\frac1k,k}$$ is meager in $Y$.
\end{corollary}

\begin{proof} To show that the set $A_{C,\alpha,\beta}$ is convex and nowhere dense, apply Corollary~\ref{cor1} to the sequences $(r_n)_{n\in\IN}$ and $(\e_n)_{n\in\IN}$, defined by $r_n=Cn^\beta$ and $\e_n=Cn^{-\alpha}$ for $n\in\IN$.

To see that $\bigcup_{C,\alpha,\beta>0}A_{C,\alpha,\beta}=\bigcup_{k\in\IN}A_{k,\frac1k,k}$, take any positive real numbers $C,\alpha,\beta$ and choose any  number $k\ge \max\{C,\frac1{\alpha},\beta\}$. The choice of $k$ guarantees that $Cn^{-\alpha}\le \kappa n^{-\frac1k}$ and $Cn^\beta\le kn^k$ for every $n\in\IN$, which implies the inclusion $A_{C,\alpha,\beta}\subset A_{k,\frac1k,k}$.
\end{proof}

\end{document}